\documentclass[12pt]{amsart}

\title[Diederich--Fornaess and Steinness indices]{Diederich--Forn{\ae}ss and Steinness indices for abstract CR manifolds}

\author{Masanori Adachi}
\address{Department of Mathematics, Faculty of Science, Shizuoka University.  836 Ohya, Suruga-ku, Shizuoka 422-8529, Japan.}
\email{adachi.masanori@shizuoka.ac.jp}
\thanks{The first author is partially supported by a JSPS KAKENHI Grant Number JP18K13422.}

\author{Jihun Yum}
\address{Department of Mathematics, Pusan National University. 2, Busandaehak-ro 63beon-gil, Geumjeong-gu, Busan 46241, Republic of Korea.}
\email{jihun0224@pusan.ac.kr}
\thanks{The second author is supported by the National Research Foundation (NRF) of Korea grant funded by the Korea government (No. 2018R1C1B3005963).}

\date{\today}

\usepackage{amsmath,amsthm,amssymb,latexsym}
\usepackage[abbrev]{amsrefs}

\usepackage{hyperref}
\usepackage[autostyle]{csquotes}
\usepackage{xcolor}

\newtheorem*{Claim}{Claim}
\newtheorem*{Question}{Question}
\newtheorem{Theorem}{Theorem}

\newtheorem{Proposition}{Proposition}[section]
\newtheorem{Definition}[Proposition]{Definition}
\newtheorem{Lemma}[Proposition]{Lemma}
\newtheorem{Corollary}[Proposition]{Corollary}

\theoremstyle{remark}
\newtheorem{Remark}[Proposition]{Remark}
\newtheorem{Example}[Proposition]{Example}

\newcommand\C{\mathbb{C}}  
\newcommand\R{\mathbb{R}}

\newcommand\D{\mathbb{D}}
  
\renewcommand\Re{\operatorname{Re}}

\newcommand{\base}[1]{\frac{\pa}{\pa #1}}

\newcommand{\pa}{\partial}
\newcommand{\opa}{\overline\pa}
\newcommand{\ol}{\overline }

\def\RR{\mathbb{R}} 
\def\CC{\mathbb{C}} 
 
\def\O{\Omega} 
\def\Lie{\mathcal{L}}  
\def\Null{\mathcal{N}}    

\begin{document}

\maketitle

\begin{abstract}
	We propose the concept of Diederich--Forn{\ae}ss and Steinness indices on compact pseudoconvex CR manifolds of hypersurface type in terms of the D'Angelo 1-form. 
	When the CR manifold bounds a domain in a complex manifold, under certain additional non-degeneracy condition, those indices are shown to coincide with the original Diederich--Forn{\ae}ss and Steinness indices of the domain, and CR invariance of the original indices follows.
\end{abstract}

\section{\bf Introduction}
In geometric analysis on complex manifolds, it is often crucial to extract plurisubharmonic 
functions from given geometric conditions. 
The celebrated example is Oka's lemma \cite{OkaVI} stating that 
the logarithm of the reciprocal of the distance to pseudoconvex hypersurfaces in Euclidean spaces is plurisubharmonic, which was an important step in the solution to the Levi problem
for unramified domains over Euclidean spaces.

Diederich and Fornaess \cite{diederich-fornaess} strengthened this result showing 
that any smoothly bounded 
pseudoconvex domain $\O$ in a Stein manifold admits negative strictly  plurisubharmonic
function in $\O$ which is bounded exhaustive. 
Roughly speaking, the \emph{Diederich-Fornaess index} of the domain
is the supremum of the H\"older exponents of these exhaustions near the boundary, 
and measures how well the pseudoconvex boundary can be approximated by strictly 
pseudoconvex hypersurfaces from inside of the domain. 

In \cite{Yum1}, the second author introduced the \emph{Steinness index} 
as its counterpart in approximations from the outside of the domain. 
Roughly speaking, the Steinness index of the domain $\O$
is the infimum of the H\"older exponents of positive strictly plurisubharmonic functions 
in $\ol{\O}^{\complement}$ which approaches to zero on the boundary $M = \partial\O$. 

For smoothly bounded pseudoconvex domains in $\C^n$, Liu \cite{Liu1} found a differential inequality on the set of weakly pseudoconvex points that expresses the Diederich--Forn{\ae}ss index of the domain. 
Exploiting Liu's idea, the second author \cite{Yum1}, \cite{Yum2} completely characterized 
both the Diederich--Forn{\ae}ss and Steinness indices of smoothly bounded pseudoconvex domains in $\C^n$ by inequalities in terms of a 1-form, called the \emph{D'Angelo 1-form} (See \S\ref{sect:d'angelo} for definition), which is a CR invariant of the boundary. 

On the other hand, the first author \cite{Adachi1}, \cite{Adachi2} studied the Diederich--Forn{\ae}ss index for smoothly bounded domains with Levi-flat boundary, and noticed that the index can be non-trivial, and is a CR invariant of the Levi-flat boundary when the normal bundle of the Levi-flat real hypersurface is positive (See \S\ref{sect:levi-flat} for its detail). Note that it is impossible for smoothly bounded domains in $\C^n$  to have entire Levi-flat boundary, hence, this kind of domains must live in complex manifolds. 

Now it is natural to ask the following question. 
\begin{Question}
Consider smoothly bounded pseudoconvex domains in complex manifolds.  
Is it true that the Diederich--Forn{\ae}ss and Steinness indices of the domain are determined by CR structure of the boundary? 
\end{Question}

To approach this question, we define the Diederich--Forn{\ae}ss and Steinness 
indices for abstract compact CR manifolds of hypersurface type
based on the formulae found in \cite{Yum2}. 
As we will see later by a simple counterexample (Example \ref{ex:counter-example}), we need to distinguish whether positivity or semi-positivity is required to define the indices, hence, we have the following 4 kinds of definitions for each index (See Definition \ref{def:DF,S indices} for precise definitions):
\begin{itemize}
\item $DF_s(\Omega)$, $S_s(\Omega)$: the Diederich--Forn{\ae}ss and Steinness indices in the strong sense for a smoothly bounded domain $\Omega$;
\item $DF_w(\Omega)$, $S_w(\Omega)$: those in the weak sense for a  smoothly bounded domain $\Omega$;
\item $DF_s(M)$, $S_s(M)$: those in the strong sense for a compact pseudoconvex CR manifold $M$ of hypersurface type;
\item $DF_w(M)$, $S_w(M)$: those in the weak sense for a compact pseudoconvex CR manifold $M$ of hypersurface type.
\end{itemize}

Our first theorem gives a general relation among these indices. 

\begin{Theorem} 
\label{thm:4 kinds of DF,S-relation}
	Let $\O$ be a relatively compact domain in a complex manifold $\widetilde{M}$ with smooth pseudoconvex boundary $M$. Then the following inequalities hold. 
	\begin{align*}
		0 \le DF_s(M) \le DF_s(\O) & \le DF_w(\O) \le DF_w(M) \le 1, \\
		1 \le S_w(M) \le S_w(\O) & \le  S_s(\O) \le S_s(M) \le \infty.	
	\end{align*}
\end{Theorem}

Our second theorem gives sufficient conditions for these 4 kinds of indices to agree, hence, answers our Question affirmatively for some typical pseudoconvex domains (See Corollary \ref{cor:typical-cases}). 

\begin{Theorem}
\label{thm:DF,S-all same}
Let $\O$ be a relatively compact domain in a complex manifold $\widetilde{M}$ with smooth pseudoconvex boundary $M$.
Assume that there exists a positive trivialization $\eta_1$ of $\C \otimes TM/T^{1,0}_M \oplus T^{0,1}_M$ that satisfies either $\opa_b \omega_{\eta_1} > 0$ on $\Null$ or $\opa_b \omega_{\eta_1} < 0$ on $\Null$. 
(For precise meaning of these notions, see \S\ref{sect:d'angelo}.) Then, it holds that 
\begin{align*}
	DF_s(M) = DF_s(\O)& =  DF_w(\O) = DF_w(M), \\
	S_w(M) = S_w(\O)& = S_s(\O) = S_s(M).
\end{align*}
\end{Theorem}

\setcounter{Theorem}{0}

Theorems \ref{thm:4 kinds of DF,S-relation} and \ref{thm:DF,S-all same} 
not only give a generalization of previously known formulae for the indices, 
but also give an alternate proof for the formulas discussed in \cite{Liu1}, \cite{Yum1} and \cite{Yum2}.
The formulae in these works were derived via the differential geometric argument of Liu \cite{Liu1}, 
and the proof relied on the fact that Euclidean metric is torsion-free and has flat curvature. 
Our proof does not use hermitian metrics nor Chern connection, and 
shows that this differential geometric property of the ambient space $\widetilde{M} = \C^n$ is not actually needed.

Also, we would like to emphasize that the Kohn weight, $\Vert z \Vert^2$ on $\C^n$, played an important role in the proof of the formula for the Diederich--Forn{\ae}ss index in \cite{Liu1}.
Theorem \ref{thm:DF,S-all same} relaxes this point by replacing the existence of strictly plurisubharmonic exhaustion on the ambient manifold $\widetilde{M}$ with non-degeneracy of $\opa_b \omega_{\eta}$, which is a condition on third order derivatives of defining functions when $M$ is realized as a real hypersurface. We do not know whether CR invariance of the Diederich--Forn{\ae}ss and Steinness indices for smoothly bounded pseudoconvex domains is true without such an additional assumption on the boundary.

The organization of this paper is as follows. We recall the definition of D'Angelo 1-form in \S\ref{sect:d'angelo}, in particular, show that it has CR invariance for the null direction of the Levi form. In \S3, we express the D'Angelo 1-form in coordinates when a CR manifold is embedded in a complex manifold as a real hypersurface. Using this expression, we prove Theorem \ref{thm:4 kinds of DF,S-relation} in \S4. In \S5, we prove Theorem \ref{thm:DF,S-all same} that gives some sufficient conditions for the indices for domains and CR manifolds to agree. In \S6, we explain the geometric meaning of our formula when our domain has Levi-flat boundary. 

\subsection*{Acknowledgement}
The authors would like to thank Kengo Hirachi for suggesting Remark \ref{rem:partially integrable}. 

\section{\bf D'Angelo 1-form}
\label{sect:d'angelo}
In this section, we recall the definition of D'Angelo 1-form, which was introduced by D'Angelo in \cite{D'Angelo0}, \cite{D'Angelo}, and developed by Boas and Straube \cite{Boas-Straube}. 

Let $M$ be an orientable smooth real manifold of dimension $2n+1$. 
We say that $M$ is an \emph{almost CR manifold of hypersurface type} 
if $M$ is equipped with a complex subbundle $T^{1,0}_M \subset \C \otimes TM$ of rank $n$ 
such that $T^{1,0}_M \cap {T^{0,1}_M} = \{ 0\}$, where $T^{0,1}_M := \ol{T^{1,0}_M}$. 
An almost CR manifold $M$ is said to be \emph{CR manifold} if the integrability condition
\begin{equation}
\label{eq:integrability}
[\Gamma(T^{0,1}_M), \Gamma(T^{0,1}_M)] \subset \Gamma(T^{0,1}_M)
\end{equation}
is fulfilled where $\Gamma$ denotes the set of smooth sections. 

Since our CR manifold $M$ is always assumed to be orientable, the bundle $\C \otimes TM/ (T^{1,0}_M \oplus T^{0,1}_M)$ is smoothly trivial, and
there exists a global trivialization of this bundle given by a purely-imaginary non-vanishing smooth 1-form $\eta$ on $M$ that annihilates $T^{1,0}_M \oplus T^{0,1}_M$. 
We call such $\eta$ as a \emph{trivialization} of $\C \otimes TM/ (T^{1,0}_M \oplus T^{0,1}_M)$ for simplicity. 
We have choices of $\eta$ up to the multiple of non-vanishing smooth real functions.

\begin{Definition}
The \emph{Levi form} at $p\in M$ is a hermitian form $\lambda_\eta \colon T^{1,0}_{M,p} \times T^{1,0}_{M,p} \to \C$ given by 
\[
\lambda_\eta(L_p, L'_p) := - d\eta(L_p, \ol{L'}_p) = \eta ([L, \ol{L'}]_p) 
\]
where $L, L' \in \Gamma(T^{1,0}_M)$ are arbitrary smooth extension of $L_p, L'_p \in T^{1,0}_{M, p}$. 
\end{Definition}

We denote the kernel of the Levi form by $\Null = \bigcup_{p \in M} \Null_p \subset T^{1,0}_M$ where
\[
\Null_p := \{ L_p \in T^{1,0}_{M,p} \mid \lambda_\eta(L_p, L'_p) = 0 \quad \forall L_p' \in T^{1,0}_{M,p} \}
\]
for each $p \in M$. Note that $\Null$ is independent of the choice of $\eta$. 
When $\lambda_\eta(L_p, L_p) \geq 0$ for any $p\in M$ and $L_p \in T^{1,0}_{M,p}$, 
$M$ is said to be \emph{pseudoconvex} and we call $\eta$ a \emph{positive} trivialization. 
The signature of $\lambda_\eta$ at any point is invariant under multiplication of $\eta$ by positive real functions,
and so is the pseudoconvexity. 
By the abuse of notation, we extend the Levi form over $T^{1,0}_M \oplus T^{0,1}_M$ as 
\[
\lambda_\eta(X_p, Y_p) := \eta ([X, \ol{Y}]_p) 
\]
for any $X, Y \in \Gamma(T^{1,0}_M \oplus T^{0,1}_M)$.
From the integrability condition (\ref{eq:integrability}) which we assumed for CR manifolds, it is clear that the following holds. 

\begin{Lemma} \label{lem:orthogonal}
	If $L_p \in \Null_p$, then $\lambda_{\eta}(L_p, X_p) = 0$ for all $X_p \in T^{1,0}_{M,p} \oplus T^{0,1}_{M,p}$.
\end{Lemma}

\begin{Remark}
\label{rem:partially integrable}
In this paper, we exploit the integrability condition (\ref{eq:integrability}) only via Lemma \ref{lem:orthogonal}. In fact, the partial integrability condition
\begin{equation}
\label{eq:partial integrability}
[\Gamma(T^{1,0}_M), \Gamma(T^{1,0}_M)] \subset \Gamma(T^{1,0}_M \oplus T^{0,1}_M)
\end{equation}
suffices to deduce Lemma \ref{lem:orthogonal}. 
An almost CR manifold $M$ is said to be a \emph{partially integrable CR manifold} if 
$M$ enjoys the partial integrability condition (\ref{eq:partial integrability}) instead of (\ref{eq:integrability}). 
We note that all the definitions and results in this section still work on partially integrable CR manifolds, and the definitions of $DF_s(M), DF_w(M), S_s(M)$ and $S_w(M)$, which will be given in Definition \ref{def:DF,S indices}, apply for compact pseudoconvex partially integrable CR manifolds of hypersurface type.
\end{Remark}

Let $T \in \Gamma(\C \otimes TM)$ be a non-vanishing smooth vector field on $M$ such that $\eta(T) = 1$. 
This complexified vector field $T$ is purely-imaginary and yields a decomposition
\[
\C \otimes TM = T^{1,0}_M \oplus T^{0,1}_M \oplus \C T.
\]
We call such $T$ as a \emph{transversal vector field} normalized with respect to $\eta$.

\begin{Definition}
A \emph{D'Angelo 1-form} $\alpha_{\eta} \colon TM \to \R$ with respect to $\eta$ is defined as a smooth real 1-form
\[
\alpha_{\eta}(X_p) := -(\mathcal{L}_T \eta)(X_p) = -d\eta(T_p, X_p) = \eta([T,X]_p),
\]
where $T$ is a transversal vector field normalized with respect to $\eta$, $X \in  \Gamma(TM)$ is an arbitrary smooth extension of $X_p \in T_pM$ and $\mathcal{L}$ denotes the Lie derivative. 
\end{Definition}

Since $\alpha_{\eta}$ is a real 1-form and enjoys $\alpha_{\eta}(T) = -d\eta(T,T) = 0$, it is natural to consider the complex linear extension of $\alpha_{\eta}$ on $\C \otimes TM$
and decompose it by its $(1,0)$-component $\omega_{\eta}$ and $(0,1)$-component $\ol{\omega}_{\eta}$, that is,
\[
\alpha_{\eta}(X_p) = \omega_{\eta}(L_p) + \ol{\omega}_{\eta}(\ol{L'}_p)
\]
where $X_p = L_p + \ol{L'}_p \in  T^{1,0}_M \oplus T^{0,1}_M$ 
and $\omega_{\eta} \colon \C \otimes TM \to \C$ annihilating $T^{0,1}_M \oplus \C T$.
We call $\omega_{\eta}$ the \emph{D'Angelo $(1,0)$-form} with respect to $\eta$.

Note that for a given trivialization $\eta$, since a transversal vector field $T$ normalized with respect to $\eta$ is not unique, $\omega_{\eta}$ depends on the choice of $T$. However, $\omega_{\eta}$ is well-defined on $\Null$ as follows.

\begin{Lemma}
\label{lem:connection}
The D'Angelo $(1,0)$-form $\omega_{\eta}$ restricted on $\Null$ is independent of the choice of $T$. 
\end{Lemma}
\begin{proof}
Let $T$ and $T'$ be transversal vector fields normalized with respect to $\eta$.
Then $X := T-T' \in \Gamma(T^{1,0}_M \oplus T^{0,1}_M)$. 
We take $L_p \in \Null_p$ and its smooth extension $L \in \Gamma(T^{1,0}_M)$. 
Then, 
\[
- (\Lie_{T}\eta)(L_p) + (\Lie_{T'}\eta)(L_p) = \eta([T, L]_p) - \eta([T',L]_p) 
= \lambda_\eta(X_p, \ol{L}_p) = 0
\]
from Lemma \ref{lem:orthogonal}.
\end{proof}

Next, we consider $(1,1)$-form $\opa_b \omega_{\eta}$ on $M$, where $\opa_b$ denotes the tangential Cauchy--Riemann operator. 
We will see its geometric meaning later in Proposition \ref{prop:brunella} and \S\ref{sect:levi-flat}.
We may identify this $(1,1)$-form with $\opa_b \omega_{\eta} \colon  T^{1,0}_M  \times T^{0,1}_M \to \C $ given by 
\[
\opa_b \omega_{\eta} (L_p, \ol{L'}_p) := \ol{L'}_p \omega_{\eta} (L) + \omega_{\eta}([L, \ol{L'}]_p)
\]
where $L, L' \in \Gamma(T^{1,0}_M)$ are arbitrary smooth extensions of $L_p, L'_p \in T^{1,0}_{M, p}$. 

\begin{Lemma}
\label{lem:curvature}
Assume that $M$ is pseudoconvex. Then, the followings hold. 
\begin{enumerate}
\item The $(1,1)$-form $\opa_b \omega_{\eta}$ induces a hermitian form on 
$\Null_p$ for each $p \in M$;
\item This hermitian form is independent of the choice of $T$.
\end{enumerate}
\end{Lemma}

The first part follows from a crucial property of D'Angelo 1-form found by Boas and Straube \cite{Boas-Straube}:

\begin{Proposition}[Boas and Straube \cite{Boas-Straube}]
\label{prop:boas-straube}
If $M$ is pseudoconvex, $d\alpha_{\eta}(X_p,Y_p) = 0$ holds for any $p \in M$ and $X_p, Y_p \in \Null_p \oplus \ol{\Null}_p$.
\end{Proposition}

\begin{Remark}
Proposition \ref{prop:boas-straube} is also true for partially integrable pseudoconvex CR manifolds because the same proof as in \cite{Boas-Straube} works. 
\end{Remark}

\begin{proof}[Proof of Lemma \ref{lem:curvature}]
\noindent (1) Let $L_p, L'_p \in \Null_p$ and take their extensions $L, L' \in \Gamma(T^{1,0}_M)$.
Then, 
\begin{align*}
d\alpha_{\eta}(L_p, \ol{L'}_p) 
&= L_p \alpha_{\eta}(\ol{L'}) - \ol{L'}_p \alpha_{\eta}(L) - \alpha_{\eta}([L, \ol{L'}]_p)\\
&= L_p \ol{\omega}_{\eta}(\ol{L'}) - \ol{L'}_p \omega_{\eta}(L) - \omega_{\eta}([L, \ol{L'}]_p)
 - \ol{\omega}_{\eta}([L, \ol{L'}]_p)\\
&= \ol{\ol{L}_p {\omega_{\eta}}({L'}) + {\omega_{\eta}}([L', \ol{L}]_p)} - (\ol{L'}_p \omega_{\eta}(L) + \omega_{\eta}([L, \ol{L'}]_p))
\\
&= \ol{\opa_b \omega_{\eta}(L'_p, \ol{L}_p)} - \opa_b \omega_{\eta}(L_p, \ol{L'}_p).
\end{align*}
This is zero from Proposition \ref{prop:boas-straube}, hence, it holds that
\[
 \opa_b \omega_{\eta}(L_p, \ol{L'}_p) = \ol{\opa_b \omega_{\eta}(L'_p, \ol{L}_p)}.
\]

\noindent (2) Let $T$ and $T'$ be transversal vector fields normalized with respect to $\eta$.
Then $X := T-T' \in \Gamma(T^{1,0}_M \oplus T^{0,1}_M)$. 
We take $L_p \in \Null_p$ and its smooth extension $L \in \Gamma(T^{1,0}_M)$. Then
\[
 \opa_b \omega_{\eta}(L_p, \ol{L'}_p) = \ol{L'}_p \omega_{\eta}(L) + \omega_{\eta} ([L, \ol{L'}]_p)
 = \ol{L'}_p \eta([T, L]) + \eta([T, [L, \ol{L'}]]_p).
\]
Therefore, we would like to show the vanishing of 
\begin{align*}
& \left( \ol{L'}_p \eta([T, L]) + \eta([T, [L, \ol{L'}]]_p) \right) - \left( \ol{L'}_p \eta([T', L]) + \eta([T', [L, \ol{L'}]]_p) \right) \\
=& \ol{L'}_p \eta([X, L]) + \eta([X, [L, \ol{L'}]]_p)\\
=& d\eta(\ol{L'}_p, [X,L]_p) + [X,L]_p \eta(\ol{L'}_p) + \eta([\ol{L'}, [X,L]]_p)+ \eta([X, [L, \ol{L'}]]_p)\\
=&  \lambda_\eta([X,L]_p, L'_p) + \eta([\ol{L'}, [X,L]]_p)+ \eta([X, [L, \ol{L'}]]_p) \\
=&  \lambda_\eta([X,L]_p, L'_p)  -\eta([L, [\ol{L'}, X]]_p)\\
=&  \lambda_\eta([X,L]_p, L'_p)  -\lambda_\eta(L_p, [{L'}, \ol{X}]_p)
\end{align*}
where we used the Jacobi identity.
Since $L_p, L'_p \in \Null_p$, this is zero by Lemma \ref{lem:orthogonal} as desired. 
\end{proof}

Now let $\widetilde{\eta}$ and $\eta$ be trivializations of $\C \otimes TM/ (T^{1,0}_M \oplus T^{0,1}_M)$. Then since those are non-vanishing, there exists a positive (or negative) smooth function $\varphi \in C^{\infty}(M)$ such that 
\[
	\widetilde{\eta} = \varphi \eta.
\]
We see the relation between D'Angelo 1-forms of $\widetilde{\eta}$ and $\eta$ as follows.

\begin{Lemma} The following relations hold.
\label{lem:two D'Angelo 1-forms-relation}
	\begin{align*}
		\omega_{\varphi \eta} &= \omega_{\eta} + \partial_b \log |\varphi| \quad \text{ on } \Null, \\
		\ol{\partial}_b \omega_{\varphi \eta} &= \ol{\partial}_b \omega_{\eta} - \partial_b \ol{\partial}_b \log |\varphi| \quad \text{ on } \Null.
	\end{align*}
\end{Lemma}
\begin{proof}
	Let $T$ be a transversal vector field normalized with respect to $\eta$. Then $\varphi^{-1} T$ is a transversal vector field normalized with respect to $\varphi \eta$. Take $X_p \in T_pM$ and its smooth extension $X \in \Gamma(TM)$. Then 
	\begin{align*}
		\alpha_{\varphi \eta}(X_p) 
		&= \varphi \eta([\varphi^{-1}T, X]_p) = \eta([T,X]_p) - \varphi (X_p \varphi^{-1}) \\
		&= \eta([T,X]_p) + \frac{X_p \varphi}{\varphi} = \alpha_{\eta}(X_p) + X_p \log|\varphi|.
	\end{align*}
	Therefore, 
	\[
		\alpha_{\varphi \eta} = \alpha_{\eta} + d \log |\varphi|,
	\]
	and this implies 
	\begin{align*}
		\omega_{\varphi \eta} &= \omega_{\eta} + \partial_b \log |\varphi|, \\
		\ol{\partial}_b \omega_{\varphi \eta} &= \ol{\partial}_b \omega_{\eta} - \partial_b \ol{\partial}_b \log |\varphi|.  
	\end{align*}
\end{proof}


\section{\bf Defining function and D'Angelo 1-form}
When our CR manifold $M$ is realized in a complex manifold $\widetilde{M}$ as a boundary of 
a domain, we can describe the D'Angelo 1-form in terms of the defining 
function $\rho$ of the domain $\Omega$. 

Let $\Omega$ be a relatively compact domain in a complex manifold $\widetilde{M}$ with smooth boundary $M$. Let $\rho$ a smooth defining function of $\Omega$, that is, 
a smooth function $\rho \colon \widetilde{M} \to \R$ satisfies $\Omega = \{ \rho < 0\}$ 
and $d\rho \neq 0$ on $M$. The CR structure of $M$ is given by $\ker \pa \rho$ and  
we can trivialize $\C \otimes TM/T^{1,0}_M \oplus T^{0,1}_M$ by $\eta_\rho := (\pa \rho - \opa \rho)/2$. Hence, for a given transversal vector field $T$ normalized with respect to $\eta_\rho$, we can write down
\[
\alpha_{\eta_\rho}(X_p) = -d\eta_\rho(T_p, X_p) = \pa\opa \rho(T_p, X_p)
\]
for $X_p \in T_pM$. We denote $\alpha_{\eta_\rho}$ and $\omega_{\eta_\rho}$ by $\alpha_{\rho}$ and $\omega_{\rho}$, respectively.

Note that for a given positive trivialization $\eta$, we can find a defining function
$\rho$ such that $\eta = \eta_\rho$ as below, and express every $\alpha_\eta$ as 
$\alpha_\rho$ using a defining function $\rho$. 

\begin{Lemma}
	For a given positive trivialization $\eta$ of $\C \otimes TM/(T^{1,0}_M \oplus T^{0,1}_M)$, there exists a smooth defining function $\rho$ of $\Omega$ such that $\eta = (\pa\rho - \opa \rho)/2$.
\end{Lemma}

\begin{proof}
Take a transversal vector field $T$ normalized with respect to $\eta$. 
Then, $iT$ is a real vector field along $M$, hence, $X := - J iT/2$ is an outward normal vector field of $M$ in $\widetilde{M}$, where $J:T\widetilde{M} \rightarrow T\widetilde{M}$ is the complex structure of $\widetilde{M}$. We extend $X$ as a non-vanishing smooth vector field on an open neighborhood of $M$ in $\widetilde{M}$, and consider the flow generated by $X$, $\Phi \colon M \times (-\delta, \delta) \to \widetilde{M}$. 
Since $X$ is non-vanishing along $M$, this map gives a diffeomorphism between $M \times (-\delta, \delta)$ and an open neighborhood $W$ of $M$ for enough small $\delta > 0$. It is clear that $\rho := \mathrm{pr}_2 \circ \Phi^{-1} \colon W \to (-\delta, \delta)$ is a defining function of $M$.

Now we claim that $\eta = \eta_\rho$ holds. Since both $\eta$ and $\eta_\rho$ are positive trivializations, there exists $\varphi \in C^{\infty}(M)$ such that $\eta = \varphi \eta_\rho$.
From the construction of $\rho$, we have $X\rho =1$ on $M$, hence, 
\[
	\eta_\rho(T) = \frac{i}{2}d^c\rho (T)
	= d\rho \left( -\frac{JiT}{2} \right) = d\rho(X) = 1,
\]
where $d^c = i(\ol{\partial}-\partial)$. Here, we used $d^c = -J^* d$ when it acts on smooth functions.  Therefore, $\varphi = \eta(T)/\eta_\rho(T) = 1$ and we conclude that $\eta = \eta_\rho$.
\end{proof}

Now assume that our $M$ is pseudoconvex and $p \in M$ is a weakly pseudoconvex point, i.e., $\Null_p \neq 0$.
We shall describe the hermitian forms $\opa_b \omega_{\rho}$ and $\omega_{\rho} \wedge \ol{\omega}_{\rho}$ on $\Null_p$ in terms of the given definition function $\rho$ by choosing a suitable local coordinate and a transversal vector field $T$. 

Take a local holomorphic coordinate $(U, z = (z', z_n) = (z_1, z_2, \dots, z_n))$, 
$z_j = x_j + i y_j$, of $\widetilde{M}$ such that $z(p) = 0$ and $(d\rho)_p = (dy_n)_0$.
Note that $T_p M \simeq \C^{n-1} \times \R$ and $T^{1,0}_{M,p} \simeq \C^{n-1} \times \{0\}$ in this local coordinate.  
The vector field
\[
T = \left(\frac{\pa \rho}{\pa z_n}\right)^{-1} \base{z_n} - \left(\frac{\pa \rho}{\pa \ol{z}_n}\right)^{-1} \base{\ol{z}_n}
\]
is well-defined on $U$ if we choose $U$ enough small. Note that 
$T$ is a transversal vector field normalized with respect to $\eta_\rho$ on $M \cap U$. 
It follows that 
\[
\widetilde{\alpha}_{\rho} := \iota_T \pa\opa \rho = \sum_{j=1}^n \frac{\pa^2 \rho}{\pa z_j \pa \ol{z}_n} \left(\frac{\pa \rho}{\pa \ol{z}_n}\right)^{-1} dz^j + \sum_{k=1}^n \frac{\pa^2 \rho}{\pa z_n \pa \ol{z}_k} \left(\frac{\pa \rho}{\pa {z_n}}\right)^{-1} d\ol{z}^k
\]
and its $(1,0)$-component
\[
\widetilde{\omega}_{\rho} = \sum_{j=1}^n \frac{\pa^2 \rho}{\pa z_j \pa \ol{z}_n} \left(\frac{\pa \rho}{\pa \ol{z}_n}\right)^{-1} dz^j  = \pa \log \left(\frac{\pa \rho}{\pa \ol{z}_n}\right)
\]
satisfy $\widetilde{\alpha}_{\rho}|TM = \alpha_{\rho}$ and $\widetilde{\omega}_{\rho}|T^{1,0}_M  = \omega_{\rho}$. 
Therefore,
\[
\opa \widetilde{\omega}_{\eta} = \opa \pa \log \left(\frac{\pa \rho}{\pa \ol{z}_n}\right)
\]
and $\opa_b \omega_{\eta} = \opa \widetilde{\omega}_{\eta} |T^{1,0}_M \times T^{0,1}_M$ from the definition of the tangential Cauchy--Riemann operator. 
We have a local description for the hermitian forms on $\Null_p$, 
\[
\opa_b \omega_{\rho} =  \opa \pa \log \left(\frac{\pa \rho}{\pa \ol{z}_n}\right), 
\quad 
\omega_{\rho} \wedge \ol{\omega}_{\rho} =  \pa \log \left(\frac{\pa \rho}{\pa \ol{z}_n}\right)
\wedge  \opa \log \left(\frac{\pa \rho}{\pa {z_n}}\right).
\]
Since Proposition \ref{prop:boas-straube} yields
\[
d\widetilde{\alpha}_{\rho} = \opa \pa  \log \left(\frac{\pa \rho}{\pa \ol{z}_n}\right)
+ \pa\opa \log \left(\frac{\pa \rho}{\pa {z_n}}\right)=0 \quad \text{on $\Null_p$}, 
\]
we see that our $\opa_b \omega_{\rho}$ actually gives a hermitian form on $\Null_p$: 
\[
\opa_b \omega_{\rho} 
=  \opa \pa \log \left(\frac{\pa \rho}{\pa {z_n}}\right)
= \frac{1}{2}  \opa \pa \log \left|\frac{\pa \rho}{\pa {z_n}}\right|^2. 
\]

Using this local description, we see a geometric interpretation of $\opa_b \omega_\rho$: 
the strict positivity of $\opa_b \omega_\rho$ corresponds to certain strong plurisubharmonicity of $-\log (-\rho)$ on $\Omega$. 

\begin{Proposition}
	\label{prop:brunella}
	Suppose that the defining function $\rho$ enjoys the strong Oka condition near $M$, namely,
	\[
	i\pa\opa (-\log (-\rho)) \geq g \quad \text{in $\Omega \cap W$}
	\]
	for some hermitian metric $g$ of $\widetilde{M}$ and an open neighborhood $W \supset M$.
	Then,  
	\[
	\opa_b \omega_{\rho} \geq g > 0 \quad\text{on $\Null$}.
	\]
\end{Proposition}

\begin{proof}
	Let $p \in M$ and $L_p \in \Null_p$.
	Take a local holomorphic coordinate $(U, z = (z', z_n) = (z_1, z_2, \dots, z_n))$, 
	$z_j = x_j + i y_j$, of $\widetilde{M}$ such that $z(p) = 0$ and $(d\rho)_p = (dy_n)_0$.
	By a change of the local coordinate, we may assume $L_p = (\base{z_1})_0$ and 
	\[
	\frac{\pa^2 \rho}{\pa z_j \pa z_k}(0) = \frac{\pa^2 \rho}{\pa \ol{z}_j \pa \ol{z}_k}(0)= 0 
	\]
	for any $j, k = 1, 2, \dots, n$ without loss of generality.
	
	From the local expression of $\opa_b \omega_{\rho}$, we have
	\[
	\opa_b \omega_{\rho}(L_p, \ol{L}_p)
	= 
	\frac{1}{2}\frac{\pa^2}{\pa z_1 \pa \ol{z}_1}( - \log \left|\frac{\pa \rho}{\pa \ol{z}_n}\right|^2)(0)
	= -\frac{\pa^3 \rho}{\pa z_1 \pa \ol{z}_1 \pa y_n}(0)
	\]
	thanks to the choice of our local coordinate.
	
	We would like to show the lower bound of this derivative from the strong Oka condition. 
	We extend $L_p$ to 
	\[
	L = \base{z_1} + \frac{\pa \rho}{\pa z_1} \left(\frac{\pa \rho}{\pa z_n}\right)^{-1} \base{z_n}
	\]
	on $U$. 
	From the assumption,
	\[
	\pa\opa (-\log (-\rho))(L_q, \ol{L}_q) \geq g(L_q, \ol{L}_q) 
	\]
	holds for any $q \in U \cap \Omega$ if we shrink $U$ so that $U \subset W$. This LHS is 
	\begin{align*}
	&\pa\opa (-\log (-\rho))(L_q, \ol{L}_q)\\
	&= \frac{1}{(-\rho)}\pa\opa \rho(L_q, \ol{L}_q) +\frac{1}{\rho^2}  \left|\pa\rho(L_q)\right|^2\\
	&= \frac{1}{(-\rho)}\left( \frac{\pa^2\rho}{\pa z_1 \pa \ol{z}_1} 
	+ 2\Re \frac{\pa^2\rho}{\pa \ol{z}_1 \pa {z_n}}  \frac{\pa \rho}{\pa z_1}   \left(\frac{\pa \rho}{\pa z_n}\right)^{-1} + \frac{\pa^2\rho}{\pa{z_n} \pa \ol{z}_n} 
	\left|\frac{\pa \rho}{\pa z_1} \right|^2  \left|\frac{\pa \rho}{\pa z_n}\right|^{-2} \right)\\
	&\quad  +\frac{4}{\rho^2}  \left|\frac{\pa \rho}{\pa z_1} \right|^2.
	\end{align*}
	Consider its value at $z = (0, \dots, 0, iy_n)$. Then, taking its limit as $y_n \nearrow 0$ yields
	\begin{align*}
	&\lim_{y_n \nearrow 0}\pa\opa (-\log (-\rho))(L_{(0, \dots, 0, iy_n)}, \ol{L}_{(0, \dots, 0, iy_n)})\\
	&=  -\frac{\pa^3\rho}{\pa z_1 \pa \ol{z}_1 \pa y_n}(0) 
	- 4\Re i \frac{\pa^2\rho}{\pa \ol{z}_1 \pa {z_n}}(0)  \frac{\pa^2 \rho}{\pa z_1 \pa y_n}(0) + 4 \left|\frac{\pa^2 \rho}{\pa z_1 \pa y_n}(0) \right|^2 \\
	&=  -\frac{\pa^3\rho}{\pa z_1 \pa \ol{z}_1 \pa y_n}(0).
	\end{align*}
	Here  we used our choice of the local coordinate
	\[
	0 = \frac{\pa^2 \rho}{\pa \ol{z}_1 \pa \ol{z}_n}(0) = \frac{1}{2} 
	\left( \frac{\pa^2 \rho}{\pa \ol{z}_1 \pa x_n}(0) + i \frac{\pa^2 \rho}{\pa \ol{z}_1 \pa y_n}(0)\right),
	\]
	hence,
	\[
	\frac{\pa^2 \rho}{\pa \ol{z}_1 \pa {z_n}}(0) = \frac{1}{2} 
	\left( \frac{\pa^2 \rho}{\pa \ol{z}_1 \pa x_n}(0) - i \frac{\pa^2 \rho}{\pa \ol{z}_1 \pa y_n}(0)\right)
	= - i \frac{\pa^2 \rho}{\pa \ol{z}_1 \pa y_n}(0).
	\]
	Therefore, we conclude
	\[
	\opa_b \omega_{\rho}(L_p, \ol{L}_p)=-\frac{\pa^3 \rho}{\pa z_1 \pa \ol{z}_1 \pa y_n}(0) \geq g(L_p, \ol{L}_p) > 0.
	\]
\end{proof}


\section{\bf Diederich--Forn{\ae}ss index and Steinness index}

In \cite{Yum2}, the second author completely characterized the Diederich--Forn{\ae}ss index and Steinness index of $\O$ in terms of D'Angelo 1-form when $\O$ is a smoothly bounded pseudoconvex domain in $\CC^n$. In this section, we would like to generalize this result for a smoothly bounded pseudoconvex domain in a complex manifold. However, it turns out that this generalization does not hold anymore and need some conditions on the boundary of the domain (see Example \ref{ex:counter-example}). For this, we introduce 4 kinds of the Diederich--Forn{\ae}ss indices and the Steinness indices, respectively. 

\begin{Definition} \label{def:DF,S indices}	
For  a relatively compact domain $\O$ in a complex manifold $\widetilde{M}$ with $C^\infty$-smooth boundary $M$,
we define the \emph{Diederich--Forn{\ae}ss index of $\Omega$ in the strong and weak sense} by 
	\begin{align*}
		DF_s(\O) &:= \sup_{\rho} \left\{ 0 < \gamma < 1 : -(-\rho)^{\gamma} \text{ is strictly plurisubharmonic on } \O \cap W \right\}, \\
		DF_w(\O) &:= \sup_{\rho} \left\{ 0 < \gamma < 1 : -(-\rho)^{\gamma} \text{ is plurisubharmonic on } \O \cap W \right\},
	\end{align*}
respectively. The \emph{Steinness index of $\Omega$ in the strong and weak sense} are defined by 
	\begin{align*}
		S_s(\O) &:= \inf_{\rho} \left\{ \gamma > 1 : \rho^{\gamma} \text{ is strictly plurisubharmonic on } \ol{\O}^{\complement} \cap W \right\}, \\
		S_w(\O) &:= \inf_{\rho} \left\{ \gamma > 1 : \rho^{\gamma} \text{ is plurisubharmonic on } \ol{\O}^{\complement} \cap W \right\},
	\end{align*}
respectively. Here $\rho$ runs all the $C^\infty$-smooth defining functions of $\O$, 
and $W$ is some open neighborhood of $M$ in $\widetilde{M}$ that can depend on $\gamma$. 

For a compact pseudoconvex CR manifold $M$ of hypersurface type, we define
the \emph{Diederich--Forn{\ae}ss index of $M$ in the strong and weak sense} by 
\begin{align*}
		DF_s(M) &:= \sup_{\eta} \left\{ 0 < \gamma < 1 : \ol{\partial}_b \omega_{\eta} - \frac{\gamma}{1-\gamma} \omega_{\eta} \wedge \ol{\omega}_{\eta} > 0 \text{ on } \Null \right\}, \\
		DF_w(M) &:= \sup_{\eta} \left\{ 0 < \gamma < 1 : \ol{\partial}_b \omega_{\eta} - \frac{\gamma}{1-\gamma} \omega_{\eta} \wedge \ol{\omega}_{\eta} \ge 0 \text{ on } \Null \right\},
\end{align*}
respectively. The \emph{Steinness index of $M$ in the strong and weak sense} are defined by 
\begin{align*}	
		S_s(M) &:= \inf_{\eta} \left\{ \gamma > 1 : -\ol{\partial}_b \omega_{\eta} - \frac{\gamma}{\gamma-1} \omega_{\eta} \wedge \ol{\omega}_{\eta} > 0 \text{ on } \Null \right\}, \\
		S_w(M) &:= \inf_{\eta} \left\{ \gamma > 1 : -\ol{\partial}_b \omega_{\eta} - \frac{\gamma}{\gamma-1} \omega_{\eta} \wedge \ol{\omega}_{\eta} \ge 0 \text{ on } \Null \right\}
	\end{align*}
where $\eta$ runs all the positive trivialization of $M$. 

If the supremum or infimum does not exist, then we define the corresponding Diederich--Forn{\ae}ss index to be 0 or the Steinness index to be $\infty$, respectively. 
\end{Definition}

\begin{Remark}
	In Definition \ref{def:DF,S indices}, we regard $\ol{\partial}_b \omega_{\eta}$ and $\omega_{\eta} \wedge \ol{\omega}_{\eta}$ as quadratic forms on $\Null$, i.e., $\ol{\partial}_b \omega_{\eta} > 0$ on $\Null$ means $\ol{\partial}_b \omega_{\eta}(L, \ol{L}) > 0$ for all $L \in \Null$.
\end{Remark}

\begin{Remark}
In the literature, we often require $-(-\rho)^\eta$  to be strict plurisubharmonic on entire domain $\Omega$ to define $DF_s(\Omega)$. 
This definition is inconvenient in the context of this paper,
since $DF_s(\Omega) = 0$ once $\Omega$ contains a compact complex analytic set of 
positive dimension, and the boundary can not control the existence of such sets; 
you may blow up any point of $\Omega$.  
In this paper, we require the strict plurisubharmonicity only near the boundary.
\end{Remark}

\begin{Remark}
From the definition, it is clear that $DF_w(M), DF_s(M), S_w(M)$ and $S_s(M)$
are CR invariant of compact pseudoconvex CR manifolds of hypersurface type.
Using this observation, it was confirmed in \cite{Yum2} that for smoothly bounded domains in $\C^n$,
$DF_w(\Omega),  DF_s(\Omega), S_w(\Omega)$ and $S_s(\Omega)$
are CR invariant of the boundary $\partial\O$. 
It is an open question whether $DF_w(\Omega),  DF_s(\Omega), S_w(\Omega)$ and 
$S_s(\Omega)$ are biholomorphism invariant of domains. 
\end{Remark}

\begin{Remark}
There is another possibility to extend the notion of these indices by relaxing boundary regularity. 
Chen \cite{Chen} introduced a Diederich--Forn{\ae}ss type index, called \emph{hyperconvexity index}, for hyperconvex domains in $\C^n$.
Harrington \cite{harrington} introduced another weaker version of the Diederich--Forn{\ae}ss index,
called the \emph{weak Diederich--Forn{\ae}ss index}, for bounded pseudoconvex domains in $\C^n$.
\end{Remark}

Our first theorem in this paper clarifies the relation that holds in general among these indices. 

\begin{Theorem} 
\label{thm:4 kinds of DF,S-relation}
	Let $\O$ be a relatively compact domain in a complex manifold $\widetilde{M}$ with smooth pseudoconvex boundary $M$. Then the following inequalities hold. 
	\begin{align*}
		0 \le DF_s(M) \le DF_s(\O)& \le DF_w(\O) \le DF_w(M) \le 1, \\
		1 \le S_w(M) \le S_w(\O) &\le S_s(\O) \le S_s(M) \le \infty.	
	\end{align*}
\end{Theorem}

The theorem will follow from the following sequence of 4 lemmas.
First, we shall show the inequalities for the indices in the weak sense
by some limiting process to the boundary. 
We begin with showing that $DF_w(\O) \le DF_w(M)$. 

\begin{Lemma}
	\label{lem:df-necessity}
	Suppose that the defining function $\rho$ enjoys $i\pa\opa (-(-\rho)^\gamma) \geq 0$ on $\O \cap W$ for some $\gamma \in (0,1)$ and open neighborhood $W \supset M$.
	Then,  
	\[
	\opa_b \omega_{\rho} - \frac{\gamma}{1-\gamma}\omega_{\rho} \wedge \ol{\omega}_{\rho} \geq 0 \quad \text{on $\Null$}.
	\]
\end{Lemma}
\begin{proof}
	Let $L_p \in \Null_p$, $p \in M$, and take the holomorphic local coordinate $(U, z)$ and vector field $L_q$ defined on $U$ as in the proof of Proposition \ref{prop:brunella}. 
	
	From the local expression of $\omega_{\rho}$
	and the computation in the proof for Proposition \ref{prop:brunella}, 
	we have
	\[
	\omega_{\rho} (L_p)
	= -2i \frac{\pa^2 \rho}{\pa z_1 \pa \ol{z}_n}(0)
	= -2i \frac{\pa^2 \rho}{\pa z_1 \pa {x_n}}(0)
	= 2 \frac{\pa^2 \rho}{\pa z_1 \pa {y_n}}(0).
	\]
	
	Consider the vector field $N = \base{z_n}$ on $U$. From the assumption, the matrix
	\[
	H := \frac{1}{\gamma (-\rho)^\gamma}
	\begin{bmatrix}
	\pa\opa(-(-\rho)^\gamma)(L_q, \ol{L}_q) & \pa\opa(-(-\rho)^\gamma)(L_q, \ol{N}_q) \\
	\pa\opa(-(-\rho)^\gamma)(N_q, \ol{L}_q) & \pa\opa(-(-\rho)^\gamma)(N_q, \ol{N}_q) \\
	\end{bmatrix}
	\]
	is positive semi-definite on $U \cap \Omega$. In particular, $\det H \geq 0$ on $U \cap \Omega$.
	
	We shall evaluate each element of $H$ at 
	$z = (0, \dots, 0, iy_n)$ and estimate its limiting behavior as $y_n \nearrow 0$. 
	From
	\[
	\frac{\pa \opa (-(-\rho)^\gamma)}{\gamma (-\rho)^\gamma}
	= \pa\opa (-\log(-\rho)) -\gamma \pa \log (-\rho) \wedge \opa \log (-\rho),
	\]
	and the computation in the proof for Proposition \ref{prop:brunella}, we have
	\[
	\lim_{y_n \nearrow 0} \frac{\pa\opa(-(-\rho)^\gamma)(L_q, \ol{L}_q) }{\gamma (-\rho)^\gamma}
	= \opa_b \omega_{\rho}(L_p, \ol{L}_p) - \gamma \omega_{\rho} \wedge \ol{\omega}_{\rho}(L_p, \ol{L}_p).
	\]
	For the off diagonal entity, we have
	\begin{align*}
	\frac{\pa\opa(-(-\rho)^\gamma)(L_q, \ol{N}_q) }{\gamma (-\rho)^\gamma}
	&= \frac{1}{(-\rho)}\pa\opa \rho(L_q, \ol{N}_q) +\frac{1-\gamma}{\rho^2}  
	\pa\rho(L_q) \opa \rho(\ol{N}_q)\\
	&= \frac{1}{(-\rho)}\left( \frac{\pa^2\rho}{\pa z_1 \pa \ol{z}_n} 
	+ \frac{\pa^2\rho}{\pa{z_n} \pa \ol{z}_n} 
	\frac{\pa \rho}{\pa z_1}   \left(\frac{\pa \rho}{\pa z_n}\right)^{-1} \right)\\
	&\quad  +2 \frac{1-\gamma}{\rho^2}  \frac{\pa \rho}{\pa z_1} \frac{\pa \rho}{\pa \ol{z}_n},
	\end{align*}
	hence,
	\begin{align*}
	\lim_{y_n \nearrow 0} (-\rho)\frac{\pa\opa(-(-\rho)^\gamma)(L_q, \ol{N}_q) }{\gamma (-\rho)^\gamma}
	&= \frac{\pa^2\rho}{\pa z_1 \pa \ol{z}_n}(0)-i (1-\gamma)  \frac{\pa^2 \rho}{\pa z_1 \pa y_n} (0)\\
	&= \frac{\gamma}{2} i\omega_{\rho}(L_p).
	\end{align*}
	For the other diagonal entity, we have
	\begin{align*}
	\frac{\pa\opa(-(-\rho)^\gamma)(N_q, \ol{N}_q) }{\gamma (-\rho)^\gamma}
	&= \frac{1}{(-\rho)}\pa\opa \rho(N_q, \ol{N}_q) +\frac{1-\gamma}{\rho^2}  
	\pa\rho(N_q) \opa \rho(\ol{N}_q)\\
	&= \frac{1}{(-\rho)} \frac{\pa^2\rho}{\pa z_n \pa \ol{z}_n} 
	+ \frac{1-\gamma}{\rho^2}  \frac{\pa \rho}{\pa z_n} \frac{\pa \rho}{\pa \ol{z}_n},
	\end{align*}
	hence,
	\begin{align*}
	\lim_{y_n \nearrow 0} \rho^2 \frac{\pa\opa(-(-\rho)^\gamma)(N_q, \ol{N}_q) }{\gamma (-\rho)^\gamma}
	&= \frac{1-\gamma}{4} .
	\end{align*}
	Therefore,
	\[
	\lim_{y_n \nearrow 0} \frac{4}{1-\gamma} \rho^2 \det H 
	= \opa_b \omega_{\rho}(L_p, \ol{L}_p)- \frac{\gamma}{1-\gamma} \omega_{\rho} \wedge\ol{\omega}_{\rho}(L_p, \ol{L}_p)\\
	\]
	must be non-negative.
\end{proof}

Similarly, the inequality $S_w(M) \leq S_w(\Omega)$ can be confirmed as follows.

\begin{Lemma}
	\label{lem:stein-necessity}
	Suppose that the defining function $\rho$ enjoys $i\pa\opa \rho^\gamma \geq 0$ on $\ol{\Omega}^{\complement} \cap W$ for some $\gamma > 1$ and open neighborhood $W \supset M$.
	Then,  
	\[
	-\opa_b \omega_{\rho} - \frac{\gamma}{\gamma-1}\omega_{\rho} \wedge \ol{\omega}_{\rho} \geq 0 \quad\text{on $\Null$}.
	\]
\end{Lemma}

\begin{proof}
	Let $L_p \in \Null_p$, $p \in M$, and take the holomorphic local coordinate $(U, z)$ and vector fields $L_q$ and $N_q$ defined on $U$ as in the proof of Proposition \ref{lem:df-necessity}.
	
	From the assumption, the matrix
	\[
	H' := \frac{1}{\gamma \rho^\gamma}
	\begin{bmatrix}
	\pa\opa \rho^\gamma(L_q, \ol{L}_q) & \pa\opa \rho^\gamma(L_q, \ol{N}_q) \\
	\pa\opa \rho^\gamma(N_q, \ol{L}_q) & \pa\opa \rho^\gamma(N_q, \ol{N}_q) \\
	\end{bmatrix}
	\]
	is positive semi-definite on $U \cap \ol{\Omega}^{\complement}$. In particular, $\det H' \geq 0$ on $U \cap \ol{\Omega}^{\complement}$.
	
	We shall evaluate each element of $H$ at 
	$z = (0, \dots, 0, iy_n)$ and estimate its limiting behavior as $y_n \searrow 0$. 
	From
	\[
	\frac{\pa \opa \rho^\gamma}{\gamma \rho^\gamma}
	= \pa\opa \log \rho +\gamma \pa \log \rho \wedge \opa \log \rho,
	\]
	and the computation in the proof for Proposition \ref{prop:brunella}, we have
	\[
	\lim_{y_n \searrow 0} \frac{\pa\opa \rho^\gamma (L_q, \ol{L}_q) }{\gamma \rho^\gamma}
	= -\opa_b \omega_{\rho}(L_p, \ol{L}_p) + \gamma \omega_{\rho} \wedge \ol{\omega}_{\rho}(L_p, \ol{L}_p).
	\]
	For the off diagonal entity, we have
	\begin{align*}
	\frac{\pa\opa \rho^\gamma (L_q, \ol{N}_q) }{\gamma \rho^\gamma}
	&= \frac{1}{\rho}\pa\opa \rho(L_q, \ol{N}_q) +\frac{\gamma-1}{\rho^2}  
	\pa\rho(L_q) \opa \rho(\ol{N}_q)\\
	&= \frac{1}{\rho}\left( \frac{\pa^2\rho}{\pa z_1 \pa \ol{z}_n} 
	+ \frac{\pa^2\rho}{\pa{z_n} \pa \ol{z}_n} 
	\frac{\pa \rho}{\pa z_1}   \left(\frac{\pa \rho}{\pa z_n}\right)^{-1} \right)\\
	&\quad  +2 \frac{\gamma-1}{\rho^2}  \frac{\pa \rho}{\pa z_1} \frac{\pa \rho}{\pa \ol{z}_n},
	\end{align*}
	hence,
	\begin{align*}
	\lim_{y_n \searrow 0} \rho \frac{\pa\opa \rho^\gamma (L_q, \ol{N}_q) }{\gamma \rho^\gamma}
	&= \frac{\pa^2\rho}{\pa z_1 \pa \ol{z}_n}(0)+i (\gamma-1)  \frac{\pa^2 \rho}{\pa z_1 \pa y_n} (0)\\
	&= \frac{ \gamma}{2} i\omega_{\rho}(L_p).
	\end{align*}
	For the other diagonal entity, we have
	\begin{align*}
	\frac{\pa\opa \rho^\gamma(N_q, \ol{N}_q) }{\gamma \rho^\gamma}
	&= \frac{1}{\rho}\pa\opa \rho(N_q, \ol{N}_q) +\frac{\gamma-1}{\rho^2}  
	\pa\rho(N_q) \opa \rho(\ol{N}_q)\\
	&= \frac{1}{\rho} \frac{\pa^2\rho}{\pa z_n \pa \ol{z}_n} 
	+ \frac{\gamma-1}{\rho^2}  \frac{\pa \rho}{\pa z_n} \frac{\pa \rho}{\pa \ol{z}_n},
	\end{align*}
	hence,
	\begin{align*}
	\lim_{y_n \searrow 0} \rho^2 \frac{\pa\opa \rho^\gamma(N_q, \ol{N}_q) }{\gamma \rho^\gamma}
	&= \frac{\gamma-1}{4} .
	\end{align*}
	Therefore,
	\[
	\lim_{y_n \searrow 0} \frac{4}{\gamma-1} \rho^2 \det H'
	= -\opa_b \omega_{\rho}(L_p, \ol{L}_p)  -\frac{\gamma}{\gamma-1}\omega_{\rho} \wedge\ol{\omega}_{\rho}(L_p, \ol{L}_p)\\
	\]
	must be non-negative.
\end{proof}

Next we shall show the inequalities for the indices in the strong sense. 
These inequalities are subtler than those in the weak sense since we have to 
preserve the strict positivity. 
Let us show the inequality $DF_s(M) \le DF_s(\O)$. 

\begin{Lemma}
	\label{lem:df-from-boundary}
	Assume that given positive trivialization $\eta$ of $\C \otimes TM/(T^{1,0}_M \oplus T^{0,1}_M)$ enjoys
	\[
	\opa_b \omega_\eta - \frac{\gamma}{1-\gamma} \omega_\eta \wedge \ol{\omega}_\eta > 0
	\quad \text{on $\Null$}.
	\] 
	Then, any smooth defining function $\rho$ of $\Omega$ such that $\eta = (\pa\rho - \opa \rho)/2$ has Diederich--Forn{\ae}ss exponent $\gamma$.
\end{Lemma}

\begin{proof}
	It is enough to check $i\pa\opa(-(-\rho)^\gamma) > 0$ locally in the following sense:
\begin{Claim}
	Let $p \in M$ and $(U,z)$ be a holomorphic local coordinate around $p$ so that 
	$z(p) = 0$, $(d\rho)_p = (dy_n)_0$ and 
	\[
	\frac{\pa^2 \rho}{\pa z_j \pa z_k} (0) = 
	\frac{\pa^2 \rho}{\pa \ol{z}_j \pa \ol{z}_k} (0) = 0
	\]
	for any $j, k = 1,2, \dots, n$.
Then, $i\pa\opa(-(-\rho)^\gamma) > 0$ holds on  $\{ (0, \dots, 0, iy_n) \mid 0 < -y_n \ll 1\}$. 
\end{Claim}
	
	Note that such a local coordinate can be taken smoothly with respect to $p \in M$,
	namely, we have a smooth map $\varphi(p,z) \colon S \times U \to \widetilde{M}$ such that 
each $\varphi(p, \cdot)\colon U \to \widetilde{M}$ gives local holomorphic coordinate around $p$ satisfying the above condition, where $p \in S \subset M$ and $U \subset \mathbb{C}^n$.
Then, we put $\widetilde{U} := \{ (p, (0,\dots,0,iy_n)) \in S \times U \mid |y_n| \ll 1\}$, which has the same real dimension as $\widetilde{M}$, and restrict $\varphi$ on $\widetilde{U}$. 
From the inverse function theorem, we see that $\varphi|\widetilde{U}$ gives a diffeomorphism from $(p,0) \in V \subset \widetilde{U}$ to $p \in \varphi(V) \subset \widetilde{M}$. The Claim states that $i\partial\overline{\partial}(-(-\rho)^\gamma)) > 0$ holds on $\varphi(V) \cap \Omega$. Our $M$ is covered by such open sets, hence, their union gives the desired $W$, for which $i\partial\overline{\partial}(-(-\rho)^\gamma)) > 0$ holds on $\Omega \cap W$.

Now we start to show Claim. 
By a linear transformation, we may assume that $\Null_p \simeq \C^r \times \{0\} \subset T^{1,0}_{M,p} \simeq \C^{n-1} \times \R$ where $r := \dim_{\CC} \Null_p$ without loss of generality. 
	We define a local frame of vector fields 
	\[
	\{(L_1)_z, \dots, (L_r)_z, (M_1)_z, \dots, (M_{n-r-1})_z, N_q \}
	\]
	on $U$ by 
	\begin{align*}
	(L_{j})_z &:= \base{z_j} + \frac{\pa \rho}{\pa z_j} \left(\frac{\pa \rho}{\pa z_n}\right)^{-1} \base{z_n}\\
	(M_{k})_z &:= \base{z_{r+k}} + \frac{\pa \rho}{\pa z_{r+k}} \left(\frac{\pa \rho}{\pa z_n}\right)^{-1} \base{z_n}\\
	N_q &:= \base{z_n}
	\end{align*}
	for $1 \leq j \leq r$, $1 \leq k \leq n-r-1$. 
	Note that $\{(L_1)_p, \dots, (L_r)_p\}$ forms a basis of $\Null_p$ 
	and the Levi form $\lambda_p$ at $p$ is positive definite on the subspace spanned by $\{(M_1)_p, \dots, (M_{n-r-1})_p\}$.

	The components of $\pa\opa (-(-\rho)^\gamma)$ at $z = (0, \dots, y_n)$ can be computed by a similar way to the proof for Proposition \ref{prop:brunella}, 
	and their boundary behavior is as follows:
	\[
	\lim_{y_n \nearrow 0} \frac{\pa\opa(-(-\rho)^\gamma)((L_j)_z, (\ol{L_k})_z) }{\gamma (-\rho)^\gamma}
	= \opa_b \omega_{\eta}((L_j)_p, (\ol{L_k})_p) - \gamma \omega_{\eta} \wedge \ol{\omega}_{\eta}((L_j)_p, (\ol{L_k})_p),
	\]
	\begin{align*}
	\lim_{y_n \nearrow 0} (-\rho)\frac{\pa\opa(-(-\rho)^\gamma)((L_j)_z, \ol{N}_q) }{\gamma (-\rho)^\gamma}
	&= \frac{\gamma}{2} i\omega_{\eta}((L_j)_p),
	\end{align*}
	for $1\leq j, k \leq r$, and 
	\begin{align*}
	\lim_{y_n \nearrow 0} \rho^2 \frac{\pa\opa(-(-\rho)^\gamma)(N_q, \ol{N}_q) }{\gamma (-\rho)^\gamma}
	&= \frac{1-\gamma}{4}.
	\end{align*}
	Also, we have
	\begin{align*}
	\lim_{y_n \nearrow 0} (-\rho)\frac{\pa\opa(-(-\rho)^\gamma)((M_j)_z, (\ol{M_k})_z) }{\gamma (-\rho)^\gamma}
	&= \frac{\pa^2 \rho}{\pa z_{r+j}\pa \ol{z}_{r+k}}(0) \\
	&= \lambda((M_j)_p, (\ol{M_k})_p),
	\end{align*}
	\begin{align*}
	\lim_{y_n \nearrow 0} (-\rho)\frac{\pa\opa(-(-\rho)^\gamma)((M_j)_z, \ol{N}_q) }{\gamma (-\rho)^\gamma}
	&= \frac{\gamma}{2} i \frac{\pa^2 \rho}{\pa z_{r+j} \pa {y_n}}(0) =: b_j
	\end{align*}
	for $1\leq j, k \leq n-r-1$, and 
	\begin{align*}
	\lim_{y_n \nearrow 0} \frac{\pa\opa(-(-\rho)^\gamma)((L_j)_z, (\ol{M_k})_z) }{\gamma (-\rho)^\gamma}
	&= - \frac{\pa^3 \rho}{\pa z_j \pa \ol{z}_{r+k} \pa y_n}(0) - 4 \gamma \frac{\pa^2 \rho}{\pa z_j \pa y_n}  {\frac{\pa^2 \rho}{\pa \ol{z}_{r+k} \pa y_n}}(0)\\
	&=: a_{j,k}
	\end{align*}
	for $1 \leq j \leq r$ and $1 \leq k \leq n-r-1$. 
	
	In short, the matrix $\widetilde{H}_z$
	representing $\pa\opa(-(-\rho)^\gamma)/\gamma (-\rho)^\gamma$ 
	with respect to the local frame
	behaves asymptotically 
	\[
	\widetilde{H}_{(0, \dots, iy_n)} \sim
	\begin{bmatrix}
	\opa_b \omega_{\eta} - \gamma \omega_{\eta} \wedge \ol{\omega}_{\eta} & A & \displaystyle \frac{i \gamma \omega_{\eta}}{2 |y_n|} \\[3ex]
	{}^t\ol{A} & \displaystyle  \frac{\lambda}{|y_n|} & \displaystyle \frac{B}{|y_n|}\\[3ex]
	\displaystyle \frac{-i\gamma \ol{\omega}_{\eta} }{2 |y_n|}&  \displaystyle \frac{{}^t\ol{B}}{|y_n|}& \displaystyle \frac{1-\gamma}{4|y_n|^2} \\
	\end{bmatrix}
	\]
	as $y_n \nearrow 0$, where $A := (a_{j,k}) \in M(r, n-r-1)$, $B := (b_j) \in \C^{n-r-1}$. 
	Since
	\[
	\opa_b \omega_\eta - \frac{\gamma}{1-\gamma} \omega_\eta \wedge \ol{\omega}_\eta > 0
	\quad \text{on $\Null_p$}
	\]
	and $\lambda_p > 0$ on $\bigoplus_{k=1}^{n-r-1} \C(M_k)_p$, 
	it follows that $\widetilde{H}_{(0, \dots, iy_n)}$ is positive definite for $0 < -y_n \ll 1$,
	and this completes the proof.
\end{proof}

A similar argument yields the inequality $S_s(\O) \le S_s(M)$. 

\begin{Lemma}
	\label{lem:stein-from-boundary}
	Assume that given positive trivialization $\eta$ of $\C \otimes TM/(T^{1,0}_M \oplus T^{0,1}_M)$ enjoys
	\[
	-\opa_b \omega_\eta - \frac{\gamma}{\gamma-1} \omega_\eta \wedge \ol{\omega}_\eta > 0
	\quad \text{on $\Null$}.
	\] 
	Then, any smooth defining function $\rho$ of $\Omega$ such that $\eta = (\pa\rho - \opa \rho)/2$ has Steinness exponent $\gamma$.
\end{Lemma}

\begin{proof}
	Consider the same local situation as in the proof of Lemma \ref{lem:df-from-boundary}.
	It is enough to show that $i\pa\opa \rho^\gamma > 0$ holds 
	on $\{ (0, \dots, 0, iy_n) \mid 0 < y_n \ll 1\}$. 
	
	From computation similar to the proof of Lemma \ref{lem:df-from-boundary}, 
	the matrix $\widetilde{H'_z}$
	representing $\pa\opa \rho^\gamma/\gamma \rho^\gamma$ 
	with respect to the local frame
	behaves asymptotically 
	\[
	\widetilde{H'}_{(0, \dots, iy_n)} \sim
	\begin{bmatrix}
	-\opa_b \omega_{\eta} + \gamma \omega_{\eta} \wedge \ol{\omega}_{\eta} & A' & \displaystyle \frac{i \gamma \omega_{\eta}}{2 |y_n|} \\[3ex]
	{}^t\ol{A'} & \displaystyle  \frac{\lambda}{|y_n|} & \displaystyle \frac{B'}{|y_n|}\\[3ex]
	\displaystyle \frac{-i\gamma \ol{\omega}_{\eta} }{2 |y_n|}&  \displaystyle \frac{{}^t\ol{B'}}{|y_n|}& \displaystyle \frac{\gamma-1}{4|y_n|^2} \\
	\end{bmatrix}
	\]
	as $y_n \searrow 0$, for some $A' \in M(r, n-r-1)$, $B' \in \C^{n-r-1}$. 
	Since
	\[
	-\opa_b \omega_\eta - \frac{\gamma}{\gamma-1} \omega_\eta \wedge \ol{\omega}_\eta > 0
	\quad \text{on $\Null_p$}
	\]
	and $\lambda_p > 0$ on $\bigoplus_{k=1}^{n-r-1} \C(M_k)_p$, 
	it follows that $\widetilde{H}_{(0, \dots, iy_n)}$ is positive definite for $0 < y_n \ll 1$,
	and this completes the proof.
\end{proof}

\begin{proof}[\bf Proof of Theorem \ref{thm:4 kinds of DF,S-relation}]
	The inequalities $DF_s(\O) \le DF_w(\O)$ and $S_w(\O) \le S_w(\O)$ are clear from their definitions.
	Moreover, $DF_w(\O) \le DF_w(M)$, $S_w(M) \le S_w(\O)$, $DF_s(M) \le DF_s(\O)$ and $S_s(\O) \le S_s(M)$ follow from Lemma \ref{lem:df-necessity}, \ref{lem:stein-necessity}, \ref{lem:df-from-boundary} and \ref{lem:stein-from-boundary}, respectively. 
\end{proof}


\section{\bf Sufficient conditions for indices to agree}

We begin this section with a simple counter-example where the indices in the strong and weak sense do not agree. 

\begin{Example}
\label{ex:counter-example}
	Let $C$ be a compact Riemann surface. Let $\widetilde{M} = C \times \CC$ and $\O = C \times \D$, where $\D$ is the unit disk. Then $\O$ is a relatively compact domain in $\widetilde{M}$ with the boundary $ M = C \times \partial \D$, which is a Levi-flat real hypersurface. 
	From the maximum principle, we see that there does not exist any strictly plurisubharmonic function on $\O \cap W$ 
	or $\ol{\O}^{\complement} \cap W$ for any neighborhood $W$ of $M$ in $\widetilde{M}$ . This implies that 
	\[
		DF_s(\O)=0 \quad \text{ and } \quad S_s(\O)=\infty.
	\]
On the other hand, for a coordinate $(z,w)$ of $\widetilde{M} = C \times \CC$, 
	\[
		\rho(z,w) = |w|^2 - 1
	\]
	is a defining function of $\O$ which is plurisubharmonic near the boundary $M$. 
	This implies that $DF_w(\O) = 1$ and $S_w(\O) = 1$.
	All together with Theorem \ref{thm:4 kinds of DF,S-relation}, we have
	\begin{align*}
		DF_s(M)=0 \le DF_s(\O)=0 & <  DF_w(\O)=1 \le DF_w(M)=1, \\
		S_w(M)=1 \le S_w(\O)=1 & < S_s(\O)=\infty \le S_s(M)=\infty.	
	\end{align*}
\end{Example}

This Example \ref{ex:counter-example} shows that 4 kinds of Diederich--Forn{\ae}ss indices (or Steinness indices) in Definition \ref{def:DF,S indices} can not be equal in general.
Our second theorem gives a sufficient condition for these Diederich--Forn{\ae}ss 
and Steinness indices agree, respectively. 

\begin{Theorem}
\label{thm:DF,S-all same}
Let $\O$ be a relatively compact domain in a complex manifold $\widetilde{M}$ with smooth pseudoconvex boundary $M$.
Assume that there exists a positive trivialization $\eta_1$ of $\C \otimes TM/T^{1,0}_M \oplus T^{0,1}_M$ that satisfies either $\opa_b \omega_{\eta_1} > 0$ on $\Null$ or $\opa_b \omega_{\eta_1} < 0$ on $\Null$. Then, it holds that 
\begin{align*}
	DF_s(M) = DF_s(\O)& =  DF_w(\O) = DF_w(M), \\
	S_w(M) = S_w(\O) &= S_s(\O) = S_s(M).
\end{align*}
\end{Theorem}

In the proof of Theorem \ref{thm:DF,S-all same}, we will show that $DF_w(M) \le DF_s(M)$ and $S_s(M) \le S_w(M)$. 
In Liu's work \cite[Lemma 2.6]{Liu1}, the Kohn weight $\|z\|^2$ on $\C^n$ was used to show the statement equivalent to $DF_w(M) \le DF_s(M)$ for smoothly bounded pseudoconvex domains in $\C^n$. Now since such a strictly plurisubharmonic function in $\widetilde{M}$ does not exist in general, 
we exploit this given $\eta_1$ instead, for which the following observation is crucial.

\begin{Lemma}
\label{lem:D'Angelo 1-form-perturbation}
	Let $\eta_0$ and $\eta_1$ be positive trivializations of $\C \otimes TM/T^{1,0}_M \oplus T^{0,1}_M$.
	Then, for all $\epsilon \in \RR$, there exists a positive trivialization $\eta_\epsilon$ of $\C \otimes TM/T^{1,0}_M \oplus T^{0,1}_M$ such that 
	\[
	\omega_{\eta_\epsilon} = (1-\epsilon) \omega_{\eta_0} + \epsilon \omega_{\eta_1}.
	\]
\end{Lemma}
\begin{proof}
	Since $\eta_0$ and $\eta_1$ are non-vanishing, there exists a non-vanishing smooth function $\varphi$ on $M$ such that $ \eta_1 = \varphi \eta_0$. Define
	\[
		\eta_\epsilon := |\varphi|^{\epsilon} \eta_0 = |\varphi|^{\epsilon} \varphi^{-1} \eta_1.
	\]
	Then, by Lemma \ref{lem:two D'Angelo 1-forms-relation},
	\begin{align*}
		\omega_{\eta_\epsilon} &= \omega_{\eta_0} + \epsilon \pa_b \log|\varphi|, \\
		\omega_{\eta_\epsilon} &= \omega_{\eta_1} + (\epsilon-1) \pa_b \log|\varphi|.
	\end{align*}
	Combining those two equations yields
	\[
		\omega_{\eta_\epsilon} = (1-\epsilon) \omega_{\eta_0} + \epsilon \omega_{\eta_1}.
	\]
\end{proof}

\begin{proof}[Proof of Theorem \ref{thm:DF,S-all same}] 
In view of Theorem \ref{thm:4 kinds of DF,S-relation}, it suffices to prove that $DF_w(M) \le DF_s(M)$ and $S_s(M) \le S_w(M)$. 
First, assume that $\opa_b \omega_{\eta_1} > 0$ on $\Null$.
We shall show $DF_w(M) \le DF_s(M)$.
Suppose that, for some $0 < \gamma_0 < 1$ and positive trivialization $\eta_0$, 
\[
	\ol{\partial}_b \omega_{\eta_0} - \frac{\gamma_0}{1 - \gamma_0} \omega_{\eta_0} \wedge \ol{\omega}_{\eta_0} \ge 0 \quad \text{ on } \Null.
\]
From Lemma \ref{lem:D'Angelo 1-form-perturbation}, there exists $\eta_\epsilon$ such that $\omega_{\eta_\epsilon} = (1-\epsilon) \omega_{\eta_0} + \epsilon \omega_{\eta_1}$ for all $0 < \epsilon < 1$. We denote $\omega_{\eta_0}$, $\omega_{\eta_1}$ and $\omega_{\eta_\epsilon}$ by $\omega_0$, $\omega_1$ and $\omega_\epsilon$, respectively. 
Then, for each $0 < \gamma < \gamma_0$,
\begin{align*}
	& \ol{\partial}_b \omega_{\epsilon} - \frac{\gamma}{1 - \gamma} \omega_{\epsilon} \wedge \ol{\omega}_{\epsilon} \\
	=& \ol{\partial}_b ((1-\epsilon) \omega_0 + \epsilon \omega_1) - \frac{\gamma}{1 - \gamma} ((1-\epsilon) \omega_0 + \epsilon \omega_1) \wedge (\ol{(1-\epsilon) \omega_0 + \epsilon \omega_1}) \\
	=& (1-\epsilon) \left( \ol{\partial}_b \omega_0 - (1-\epsilon) \frac{\gamma}{1-\gamma} \omega_0 \wedge \ol{\omega}_0 \right)
	+ \epsilon \left( \ol{\partial}_b \omega_1 - \epsilon \frac{\gamma}{1-\gamma} \omega_1 \wedge \ol{\omega}_1 \right) \\
	& - \frac{\gamma}{1-\gamma}\epsilon(1-\epsilon) \left( \omega_0 \wedge \ol{\omega}_1 + \omega_1 \wedge \ol{\omega}_0 \right) \\
	\ge& (1-\epsilon) \left( \ol{\partial}_b \omega_0 - (1-\epsilon) \frac{\gamma}{1-\gamma} \omega_0 \wedge \ol{\omega}_0 \right)
	+ \epsilon \left( \ol{\partial}_b \omega_1 - \epsilon \frac{\gamma}{1-\gamma} \omega_1 \wedge \ol{\omega}_1 \right) \\
	& - \frac{\gamma}{1-\gamma}\epsilon(1-\epsilon) \left( \frac{1}{\sqrt{\epsilon}}\omega_0 \wedge \ol{\omega}_0 + \sqrt{\epsilon}\omega_1 \wedge \ol{\omega}_1 \right) \\
	=& (1-\epsilon) \left( \ol{\partial}_b \omega_0 - (1-\epsilon+\sqrt{\epsilon}) \frac{\gamma}{1-\gamma} \omega_0 \wedge \ol{\omega}_0 \right) \\
	& + \epsilon \left( \ol{\partial}_b \omega_1 - (\epsilon + (1-\epsilon)\sqrt{\epsilon}) \frac{\gamma}{1-\gamma} \omega_1 \wedge \ol{\omega}_1 \right).
\end{align*}
Now since $\gamma < \gamma_0$, if we choose sufficiently small $\epsilon>0$, then 
\[
(1-\epsilon+\sqrt{\epsilon}) \frac{\gamma}{1-\gamma} < \frac{\gamma_0}{1-\gamma_0}.
\]
This implies that
\[
	\ol{\partial}_b \omega_0 - (1-\epsilon+\sqrt{\epsilon}) \frac{\gamma}{1-\gamma} \omega_0 \wedge \ol{\omega}_0 \ge 0 \quad \text{ on } \Null.
\]
Also, the assumption $\ol{\partial}_b \omega_1 > 0$ yields that 
\[
	\ol{\partial}_b \omega_1 - (\epsilon + (1-\epsilon)\sqrt{\epsilon}) \frac{\gamma}{1-\gamma} \omega_1 \wedge \ol{\omega}_1 > 0 \quad \text{ on } \Null.
\]
for sufficiently small $\epsilon > 0$. Therefore,
\[
	\ol{\partial}_b \omega_{\epsilon} - \frac{\gamma}{1 - \gamma} \omega_{\epsilon} \wedge \ol{\omega}_{\epsilon} > 0 \quad \text{ on } \Null
\]
holds for some $\epsilon > 0$. This shows that $DF_w(M) \le DF_s(M)$.

The inequality $S_s(M) \le S_w(M)$ can be shown in a similar manner.
Suppose that, for some $\gamma_0 > 1$ and positive trivialization $\eta_0$, 
\[
-\ol{\partial}_b \omega_{\eta_0} - \frac{\gamma_0}{\gamma_0  - 1} \omega_{\eta_0} \wedge \ol{\omega}_{\eta_0} \ge 0 \quad \text{ on } \Null.
\]
From Lemma \ref{lem:D'Angelo 1-form-perturbation}, there exists $\eta_\epsilon$ such that $\omega_{\eta_\epsilon} = (1-\epsilon) \omega_{\eta_0} + \epsilon \omega_{\eta_1}$ for all $-1 < \epsilon < 0$. We denote $\omega_{\eta_0}$, $\omega_{\eta_1}$ and $\omega_{\eta_\epsilon}$ by $\omega_0$, $\omega_1$ and $\omega_\epsilon$, respectively. 
Then, for each $\gamma > \gamma_0$, it follows that 
\begin{align*}
	& -\ol{\partial}_b \omega_{\epsilon} - \frac{\gamma}{\gamma-1} \omega_{\epsilon} \wedge \ol{\omega}_{\epsilon} \\
	\ge & (1-\epsilon) \left( -\ol{\partial}_b \omega_0 - (1-\epsilon+\sqrt{-\epsilon}) \frac{\gamma}{\gamma-1} \omega_0 \wedge \ol{\omega}_0 \right) \\
	& + (-\epsilon) \left( \ol{\partial}_b \omega_1 + (\epsilon - (1-\epsilon)\sqrt{-\epsilon}) \frac{\gamma}{\gamma-1} \omega_1 \wedge \ol{\omega}_1 \right) > 0
\end{align*}
on $\Null$ if we choose $\epsilon < 0$ sufficiently closed to zero.

\bigskip

Now assume that $\opa_b \omega_{\eta_1} < 0$ on $\Null$. 
In a similar method as above, $DF_w(M) \le DF_s(M)$ can be shown by choosing $\epsilon < 0$ sufficiently closed to zero, 
and $S_s(M) \le S_w(M)$ can be shown by choosing sufficiently small $\epsilon > 0$.

\end{proof}

A criterion for $DF_s(\O) > 0$ and $S_s(\O) < \infty$ follows from the definitions of $DF_s(M)$ and $S_s(M)$.

\begin{Corollary}
\label{cor:existence of indices}
Let $\O$ be a relatively compact domain in a complex manifold $\widetilde{M}$ with smooth pseudoconvex boundary $M$.
\begin{enumerate}
\item Assume that there exists a positive trivialization $\eta$ of $\C \otimes TM/T^{1,0}_M \oplus T^{0,1}_M$ such that $\opa_b \omega_{\eta} > 0$ on $\Null$. Then, it holds that 
\[
0 < DF_s(M) = DF_s(\O) =DF_w(\O) = DF_w(M).
\]
\item Assume that there exists a positive trivialization $\eta$ of $\C \otimes TM/T^{1,0}_M \oplus T^{0,1}_M$ such that $\opa_b \omega_{\eta} + \omega_{\eta} \wedge \ol{\omega}_{\eta} < 0$ on $\Null$. Then, it holds that 
\[
S_w(M) = S_w(\O) = S_s(\O) = S_s(M) < \infty.
\]
\end{enumerate}
\end{Corollary}

\begin{proof}
Since $\opa_b \omega_{\eta} + \omega_{\eta} \wedge \ol{\omega}_{\eta} < 0$ on $\Null$ implies that $\opa_b \omega_{\eta} < 0$ on $\Null$,
Theorem \ref{thm:DF,S-all same} yields the both equalities of indices. 
When $\opa_b \omega_{\eta} > 0$ on $\Null$, it follows that for sufficiently small $\gamma > 0$, 
\[
\opa_b \omega_{\eta} - \frac{\gamma}{1-\gamma} \omega_{\eta} \wedge \ol{\omega}_{\eta} > 0
\]
holds on $\Null$ because our $M$ is compact, hence, $DF_s(M) > 0$.

When $\opa_b \omega_{\eta} + \omega_{\eta} \wedge \ol{\omega}_{\eta} < 0$ on $\Null$, it follows that for sufficiently large $\gamma > 1$, 
\[
\opa_b \omega_{\eta} + \frac{\gamma}{\gamma-1} \omega_{\eta} \wedge \ol{\omega}_{\eta}
= \left(\opa_b \omega_{\eta} + \omega_{\eta}\wedge \ol{\omega}_{\eta} \right)-  \frac{1}{\gamma-1} \omega_{\eta} \wedge \ol{\omega}_{\eta} < 0
\]
holds on $\Null$ because our $M$ is compact, hence, $S_s(M) < \infty$.
\end{proof}

We give some classes of weakly pseudoconvex domains that satisfy $\opa_b \omega_{\eta} > 0$ on $\Null$.
This gives a generalization of the result \cite[Theorem 1.1]{Yum2} obtained by the second author previously.

\begin{Corollary} \label{cor:main cor 2}
\label{cor:typical-cases}
Let $\O$ be a relatively compact domain in a complex manifold $\widetilde{M}$ with smooth pseudoconvex boundary $M$.
Then
\begin{align*}
	0 < DF_s(M) = DF_s(\O) & =   DF_w(\O) = DF_w(M), \\
	S_w(M) = S_w(\O) & = S_s(\O) = S_s(M)
\end{align*}
if one of the following conditions is satisfied: 
\begin{enumerate}
\item The ambient manifold $\widetilde{M}$ is Stein, in particular, when $\widetilde{M} = \C^n$;
\item The domain $\O$ is Takeuchi 1-convex;
\item The boundary $M$ has positive Diederich--Forn{\ae}ss index in the strong sense, i.e., $DF_s(M) > 0$.
\end{enumerate}
\end{Corollary}

Here Takeuchi 1-convexity is defined as follows.

\begin{Definition}[\cite{diederich-ohsawa}]
\label{def:takeuchi}
Let $\Omega$ be a relatively compact domain in a complex manifold $\widetilde{M}$ with smooth boundary $M$.
The domain $\Omega$ is said to be \emph{Takeuchi 1-convex} if there exists a smooth defining function $\rho \colon \widetilde{M} \to \R$ of $\Omega$
satisfying the strong Oka condition near $M$: 
\[
i\pa\opa (-\log (-\rho)) \geq g \quad \text{on $\Omega \cap W$}
\]
for some hermitian metric $g$ on $\widetilde{M}$ and an open neighborhood $W \supset M$.
\end{Definition}

\begin{proof}[Proof of Corollary \ref{cor:main cor 2}]
(1) When $\widetilde{M}$ is Stein, we have a strictly plurisubharmonic exhaustion function $\varphi$ on $\widetilde{M}$.
For any trivialization $\eta$ of $\C \otimes TM/T^{1,0}_M \oplus T^{0,1}_M$, define $\widetilde{\eta} := \exp(r \varphi) \eta$ for some constant $r \in \RR$. Then by Lemma \ref{lem:two D'Angelo 1-forms-relation}, 
\[
	\opa_b \omega_{\widetilde{\eta}} = \opa_b \omega_{\eta} - r \pa_b \opa_b \varphi.
\]
Since the boundary $\pa\O$ is compact and $i \pa\opa\varphi$ is positive on $\widetilde{M}$, one can have $\opa_b \omega_{\widetilde{\eta}} > 0$ (or $\opa_b \omega_{\widetilde{\eta}} < 0 $) on $\Null$ by choosing sufficiently negative $r<0$ (or sufficiently positive $r>0$, respectively).

(2) When $\Omega$ is Takeuchi 1-convex, the existence of a positive trivialization $\eta$ with $\opa_b \omega_{\eta} > 0$ on $\Null$
follows from Proposition \ref{prop:brunella}, hence, the conclusion follows. 
In fact, this argument gives another proof for the first case 
since any smoothly bounded pseudoconvex domain $\O$ in a Stein manifold 
is Takeuchi 1-convex (see Harrington and Shaw \cite[Theorem 1.4]{harrington-shaw}; cf.  Diederich and Forn{\ae}ss \cite{diederich-fornaess}, and Ohsawa and Sibony \cite{ohsawa-sibony}).

(3) When $DF_s(M) > 0$, the definition of the Diederich--Forn{\ae}ss index in the strong sense implies the existence of positive trivialization $\eta$ such that $\opa_b \omega_\eta > 0$ on $\Null$.
\end{proof}

\begin{Remark}
\label{rem:gap adachi}
Takeuchi 1-convexity is much weaker than being in a Stein manifold. 
Actually, Diederich and Ohsawa \cite{diederich-ohsawa} observed that certain 
smoothly bounded domain with Levi-flat boundary, which cannot be realized in a Stein manifold due to the maximum principle, can be Takeuchi 1-convex. 
By Ohsawa and Sibony \cite{ohsawa-sibony}, Takeuchi 1-convexity is a sufficient condition for $\Omega$ to have positive Diederich--Forn{\ae}ss index in the strong sense.
We do not know whether it is also a necessary condition, although the first author had claimed it in \cite[Theorem 2.4]{Adachi1} by error. 
We refer the reader to a recent survey by Fu and Shaw \cite{fu-shaw}.
\end{Remark}

Before closing this section, we remark that an embedded compact complex manifold gives an obstruction for the positivity or negativity of $\opa_b \omega_{\eta}$ on $\Null$.

\begin{Proposition}
Let $M$ be a compact pseudoconvex CR manifold of hypersurface type. 
Assume that there is a CR embedding $\iota \colon A \to M$ of a compact complex manifold $A$ of positive dimension.
Then, there is no positive trivialization $\eta$ such that $\opa_b \omega_{\eta} > 0 $ on $\Null|_{\iota(A)}$ nor $\opa_b \omega_{\eta} < 0 $ on $\Null|_{\iota(A)}$.
In particular, $DF_s(M) = 0$ and $S_s(M) = \infty$. 
\end{Proposition}

\begin{proof}
Suppose the contrary that we have a positive trivialization $\eta$ 
such that $\opa_b \omega_{\eta} > 0$ on $\Null|_{\iota(A)}$.
Then, $\beta := \iota^*\omega_{\eta}$ is a smooth $(1,0)$-form on $A$ 
such that $i\opa \beta$ is a positive $(1,1)$-form on $A$. 
Proposition \ref{prop:boas-straube} yields $d\iota^* \alpha_\eta = 0$,
hence $\pa \beta = 0$ follows. 
Notice that 
$d + \beta$ defines a connection of the trivial $\C$-line bundle over $A$
and its curvature 2-form is $d\beta = \opa\beta$. 
This means that the first Chern class of the trivial $\C$-line bundle
is represented by a positive $(1,1)$-form $i\opa\beta/2\pi$.
This is a contradiction since $ 0 \neq \int_A (i\opa\beta/2\pi)^{\dim_\C A}$.

By the same argument, we see that there is no positive trivialization $\eta$ such that $\opa_b \omega_{\eta} < 0 $ on $\Null|_{\iota(A)}$.
\end{proof}


\section{\bf Levi-flat case}
\label{sect:levi-flat}

In this section, we illustrate our formula in smoothly bounded domain $\Omega$ with Levi-flat boundary $M$. Recall that a CR manifold $M$ of hypersurface type is said to be \emph{Levi-flat}
if $\Null = T^{1,0}_M$, i.e., the Levi form $\lambda_\eta$ vanishes identically on $T^{1,0}_M$. 
This condition is equivalent for $\Re T^{1,0}_M \subset TM$ to be integrable, and
Frobenius' theorem yields a smooth foliation of real codimension one, 
which is called the \emph{Levi foliation} of $M$. 
Each leaf of the Levi foliation is endowed with an integral complex structure by the CR structure $T^{1,0}_M$.
In short, a Levi-flat CR manifold of dimension $2n+1$ is a manifold foliated by complex manifolds of complex dimension $n$. 

For a Levi-flat CR manifold $M$, the bundle $\C \otimes TM/(T^{1,0}_M \oplus T^{0,1}_M)$ is called the \emph{normal bundle} of the Levi foliation, and we denote it by $N^{1,0}_M$. 
We fix a foliated atlas of $M$, say, $\{ (U_\mu, \varphi_\mu) \}$ where $\varphi_\mu = (z_\mu, t_\mu)\colon U_\mu  \to \D^n \times (-1,1)$ are positively oriented foliated charts. 
These coordinates change in the way that 
\[
z_\mu = z_\mu(z_\nu, t_\nu), \quad t_\mu = t_\mu(t_\nu)
\]
on $U_\mu \cap U_\nu$ and $z_\mu$ is holomorphic in $z_\nu$. 
Using these charts, $N^{1,0}_M$ is locally trivialized by $\pa/\pa t_\mu$ on $U_\mu$ and the transition function from $U_\nu$ to $U_\mu$ is given by $d t_\mu/d t_\nu$, which is constant on each plaque $\varphi_\nu^{-1}(\D^n \times \{ t_\nu \})$. 
Hence, the restriction of $N^{1,0}_M$ on each leaf is an $\R^*$-flat holomorphic line bundle.

We endow a hermitian metric $h$ of $N^{1,0}_M$. 
Since the restriction of $N^{1,0}_M$ on each leaf is a holomorphic line bundle, 
we may consider the Chern connection and its curvature for $N^{1,0}_M$ with respect to $h$ leafwise. 
We call them \emph{leafwise Chern connection} and \emph{leafwise curvature} respectively. 
Note that since the restriction of $N^{1,0}_M$ on each leaf is $\R^*$-flat, 
the connection form is globally well-defined leafwise $(1,0)$-form. 

We express the hermitian metric locally by 
\[
h_\mu = h\left(\frac{\pa}{\pa t_\mu}, \frac{\pa}{\pa t_\mu}\right)
\]
on $U_\mu$, and define $\eta := i\sqrt{h_\mu}dt_\mu$. Then, $\eta$ is a well-defined purely-imaginary 1-form on $M$ and gives a (positive) trivialization of $N^{1,0}_M$.
Conversely, a trivialization $\eta$ of $N^{1,0}_M$ gives a hermitian metric $h_\eta$ of $N^{1,0}_M$ by $h_\eta \colon N^{1,0}_M \to [0,\infty)$, $h_\eta(v) := |\eta(v)|^2$. 

\begin{Lemma}
The D'Angelo $(1,0)$-form $\omega_\eta$ agrees with the connection form of the leafwise Chern connection for $N^{1,0}_M$ up to a multiplicative constant.
\end{Lemma}

\begin{proof}
We take a transverse vector field on $U_\mu$, 
\[
T_\mu = \frac{-i}{\sqrt{h_\mu}}\frac{\pa}{\pa t_\mu},
\]
normalized with respect to $\eta$. Thanks to Lemma \ref{lem:connection}, the D'Angelo $(1,0)$-form is 
\begin{align*}
\omega_\eta(X_p) & = \eta([T_\mu, X]_p)
= i \sum_{j=1}^n X^j \frac{\pa}{\pa z_\mu^j} \left(\frac{1}{\sqrt{h_\mu}}\right) \eta\left(\frac{\pa}{\pa t_\mu}  \right)\\
& =  \frac{1}{2} \sum_{j=1}^n X^j \frac{1}{h_\mu}\frac{\pa h_\mu}{\pa z_\mu^j} 
 =  \frac{1}{2} \pa_b \log h_\mu (X_p)
\end{align*}
for $p \in U_\mu$ and $X = \sum X^j \pa/\pa z_\mu^j \in \Gamma(T^{1,0}_M)$. 
Since the connection form is $\pa_b \log h_\mu$, 
the proof is completed. 
\end{proof}

\begin{Remark}
In \cite{Adachi1}, \cite{Adachi2} and \cite{Adachi-Brinkschulte}, 
D'Angelo $(1,0)$-form has already appeared as the connection form for the normal bundle of Levi foliation. 
For instance, the leafwise $(1,0)$-form $\alpha$ in \cite[\S2.4]{Adachi-Brinkschulte}
agrees with D'Angelo $(1,0)$-form, $\alpha =\omega_\eta$, where $\eta$ is the trivialization induced from given defining function of a Levi-flat real hypersurface.
\end{Remark}

Our main theorem recovers the formula for the Diederich--Forn{\ae}ss index of domains with Levi-flat boundary, which was claimed by the first author in \cite[Theorem 1.1]{Adachi1} (See Remark \ref{rem:gap adachi} for the gap in the proof), under an additional assumption that the positivity of the normal bundle.
\begin{Corollary}
\label{cor:levi-flat DF}
Let $\Omega$ be a relatively compact domain in a complex surface
with smooth Levi-flat boundary $M$. 
Assume that the normal bundle $N^{1,0}_M$ of the Levi foliation on $M$ admits a 
smooth hermitian metric with positive leafwise curvature.
Then, it holds that 
\[
	0 < DF_s(M) = DF_s(\O) =  DF_w(\O) = DF_w(M).
\]
\end{Corollary}

\begin{proof}
We define $\eta$ using the hermitian metric $h$ of $N^{1,0}_M$ with positive leafwise curvature. Then, since the $(1,1)$-form $\opa_b \omega_\eta$ is exactly the half 
of the leafwise curvature form $i\opa_b \pa_b \log h$ of $N^{1,0}_M$, which is assumed to be positive, Theorem \ref{thm:DF,S-all same} applies.
\end{proof}

\begin{Remark}
\label{rem:brunella}
Brunella \cite[\S3.3]{Brunella} gave an example of $\Omega$ with $\dim_{\CC}\O = 2$ enjoying the conditions in Corollary \ref{cor:levi-flat DF}, and we can actually compute the exact value of $DF_s(\Omega)$ of this example based on Corollary \ref{cor:levi-flat DF}. See \cite[Example 4.5]{Adachi1} for its detail.  
\end{Remark}

When $\dim_{\CC} \Omega \geq 3$, Corollary \ref{cor:levi-flat DF} is still logically correct, 
but there is no such example due to the following theorem of Brinkschulte \cite{Brinkschulte}:

\begin{quote}
Let $\widetilde{M}$ be a complex manifold of dimension $\geq 3$. 
Then, there does not exist a compact Levi-flat real hypersurface $M$ in $\widetilde{M}$ 
whose normal bundle admits a smooth hermitian metric with positive leafwise curvature.
\end{quote}

\begin{Corollary}
Let $\Omega$ be a relatively compact domain with smooth Levi-flat boundary $M$
in a complex manifold of dimension $\geq 3$. 
Then, $DF_s(M) = 0$.
\end{Corollary}

\begin{proof}
If $DF_s(M) > 0$, then there must be a trivialization $\eta$ such that $\opa_b \omega_\eta > 0$. However, this requires a hermitian metric $h_\eta$ of the normal bundle with positive leafwise curvature, and this is impossible by Brinkschulte's theorem. 
\end{proof}

We do not know whether we 
can relax the assumption in Corollary \ref{cor:levi-flat DF} to the weaker condition that $DF_s(\Omega) > 0$ as the first author had claimed in \cite[Theorem 1.1]{Adachi1}. 
We also do not know whether there is a smoothly bounded domain $\Omega$ with Levi-flat boundary and $DF_s(\Omega) > 0$ in a complex manifold of dimension $\geq 3$.

For Steinness indices, we have some restrictions to their values for compact Levi-flat CR manifolds. 
First, we observe that Steinness indices in the strong sense must diverge for domains with Levi-flat boundary. 

\begin{Proposition}
Let $\Omega$ be a relatively compact domain in a complex manifold 
with smooth Levi-flat boundary $M$. 
Assume that $\dim_{\CC} \Omega \geq 2$. 
Then, it holds that 
\[
S_s(\O) = S_s(M) = \infty.
\]
\end{Proposition}

\begin{proof}
Since $S_s(\O) \leq S_s(M)$ from Theorem \ref{thm:4 kinds of DF,S-relation}, it suffices to show $S_s(\O) = \infty$.
Suppose that $S_s(\O) < \infty$. We have a defining function $\rho$ of $\O$ such that 
$\rho^\gamma$ is strictly plurisubharmonic on $\ol{\O}^{\complement} \cap W$
for some $\gamma > 1$ and $W \supset M$. 
Then, for some $\delta > 0$, the sublevel set $\Omega_\delta := \rho^{-1}(-\infty, \delta)$
is contained in $\Omega \cup W$ and has smooth strictly pseudoconvex boundary. 
From Grauert's solution to the Levi problem on complex manifolds \cite{grauert1, grauert2}, 
$\Omega_\delta$ is a proper modification of a Stein space, hence, 
cannot contain a closed Levi-flat real hypersurface $M$. 
This is a contradiction.
\end{proof}

We shall show that $S_s(M) = \infty$ is still true for an abstract Levi-flat CR manifold $M$
adapting the argument in \cite[Lemma 1]{Adachi2}.

\begin{Lemma}
\label{lem:stokes}
Let $M$ be a compact Levi-flat CR manifold of $\dim_{\R} M = 2n+1 \geq 3$,
and $\eta$ a trivialization of $N^{1,0}_M$. If $\opa_b \omega_\eta - \frac{1}{n}\omega_\eta \wedge \ol{\omega}_\eta \leq 0$ everywhere, then, $(\opa_b \omega_\eta - \frac{1}{n}\omega_\eta \wedge \ol{\omega}_\eta)^n = 0$ everywhere.
\end{Lemma}

\begin{proof}
Suppose that $\opa_b \omega_\eta - \frac{1}{n}\omega_\eta \wedge \ol{\omega}_\eta \leq 0$ holds everywhere. Then, $(-1)^n(\opa_b \omega_\eta - \frac{1}{n} \omega_\eta \wedge \ol{\omega}_\eta)^n$ is a semi-positive leafwise $(n,n)$-form. On the other hand, the same argument as in \cite[Lemma 1]{Adachi2} yield
\begin{align*}
\int_M (\opa_b \omega_\eta - \frac{1}{n} \omega_\eta \wedge \ol{\omega}_\eta)^n \wedge \eta
& = \int_M d \left((\opa_b \omega_\eta - \frac{1}{n} \omega_\eta \wedge \ol{\omega}_\eta)^{n-1} \wedge \omega_\eta \wedge \eta\right)\\
& = 0.
\end{align*}
Hence, $(\opa_b \omega_\eta - \frac{1}{n} \omega_\eta \wedge \ol{\omega}_\eta)^n = 0$ everywhere.
\end{proof}

\begin{Proposition}
Let $M$ be a compact Levi-flat CR manifold of $\dim_{\R} M = 2n+1 \geq 3$. 
Then, $S_s(M) = \infty$. 
\end{Proposition}

\begin{proof}
Suppose that $S_s(M) < \infty$. Then, from the definition, we have a trivialization $\eta$ of $N^{1,0}_M$ 
such that $\opa_b \omega_\eta + \omega_\eta \wedge \ol{\omega}_\eta < 0$,
which implies $\opa_b \omega_\eta - \frac{1}{n}\omega_\eta \wedge \ol{\omega}_\eta < 0$ everywhere. This contradicts to Lemma \ref{lem:stokes}.
\end{proof}

\begin{Remark}
	The same argument above yields that the normal bundle of compact Levi-flat CR manifold $M$ of $\dim_{\R} M \geq 3$ does not admit a hermitian metric $h_\eta$ of negative leafwise curvature, i.e., 
	$\opa_b \omega_\eta < 0$. 
	This shows that, in view of Brunella's example (see Remark \ref{rem:brunella}), two sufficient conditions in Theorem \ref{thm:DF,S-all same} are not equivalent in general. 
\end{Remark}

A similar argument yields a dichotomy for $S_w(M)$.

\begin{Proposition}
Let $M$ be a compact Levi-flat CR manifold of $\dim_{\R} M = 3$. 
Then, we have either $S_w(M) = 1$ or $\infty$. 
\end{Proposition}

\begin{proof}
Suppose that $S_w(M) < \infty$. Then, from the definition, we have a trivialization $\eta$ of $N^{1,0}_M$ 
such that $\opa_b \omega_\eta + \omega_\eta \wedge \ol{\omega}_\eta \leq 0$,
which implies $\opa_b \omega_\eta - \omega_\eta \wedge \ol{\omega}_\eta \leq 0$. It follows from Lemma \ref{lem:stokes} that 
$\opa_b \omega_\eta - \omega_\eta \wedge \ol{\omega}_\eta = 0$ everywhere, hence,
$0 \leq 2 \omega_\eta \wedge \ol{\omega}_\eta  = \opa_b \omega_\eta + \omega_\eta \wedge \ol{\omega}_\eta \leq 0$, that is, $\omega_\eta = 0$. 
Using this $\eta$, we deduce that $S_w(M) = 1$. 
\end{proof}

\end{document}